\documentclass[11pt]{article}

\usepackage{amsmath,amssymb,verbatim}

\usepackage{times}

\usepackage{amsfonts}

\usepackage[dvips]{graphicx}

\usepackage{amsfonts}

\usepackage{amsthm}

\usepackage{epsfig}

\usepackage{graphics}

\usepackage{color}

\usepackage[utf8]{inputenc}

\usepackage{tikz}

\usepackage{amsthm}

\usepackage{setspace}

\theoremstyle{plain}

\newtheorem{theorem}{Theorem}[section]

\newtheorem{lem}[theorem]{Lemma}

\newtheorem{prop}[theorem]{Proposition}

\newtheorem{col}[theorem]{Corollary}

\theoremstyle{definition}

\newtheorem{defn}[theorem]{Definition}

\newtheorem{note}[theorem]{Note}

\newtheorem{example}[theorem]{Example}

\newtheorem{notation}[theorem]{Notation}

\newcommand{\XX}{{\mathcal{X}}}

\newcommand{\YY}{{\mathcal{Y}}}

\newcommand{\EE}{{\mathcal{E}}}

\newcommand{\CC}{{\mathcal{C}}}

\newcommand{\D}{\Delta}

\newcommand{\st}{\ : \ }

\newcommand{\ver}{\mbox{Vert}}

\newcommand{\void}[1]{}

\newcommand{\idiot}[1]{\vspace{5 mm}\par \noindent

\marginpar{\textsc{Note}}

\framebox{\begin{minipage}[c]{0.95 \textwidth}

\tt #1 \end{minipage}}\vspace{5 mm}\par}

\renewcommand{\idiot}[1]{}

\numberwithin{equation}{section}

\title{On the resolution of path ideals of cycles}

\begin{document}

\author{Ali Alilooee\thanks{Department of Mathematics and Statistics, Dalhousie University, Halifax, Canada, alilooee@mathstat.dal.ca.}
\and Sara Faridi\thanks{Department of Mathematics and Statistics, Dalhousie University, Halifax, Canada, faridi@mathstat.dal.ca.}}

\maketitle 
\begin{abstract}
We give a formula to compute all the top degree graded Betti numbers
of the path ideal of a cycle. Also we will find a criterion to
determine when Betti numbers of this ideal are non zero and give a
formula to compute its projective dimension and regularity.
\end{abstract}


\section{Introduction}


Path ideals of graphs were first introduced by Conca and De
Negri~\cite{Conca1999} in the context of monomial ideals of linear
type.  Simply put, path ideals are ideals whose monomial generators
correspond to vertices in paths of a given length in a graph. Conca
and De Negri showed that path ideals of trees have normal and
Cohen-Macaulay Rees rings. More recently Bouchat, H\`{a} and
O'Keefe~\cite{R.Bouchat2010} and He and Van Tuyl~\cite{He2010} studied
invariants related to resolutions of path ideals of certain graphs.
In his thesis, Jacques~\cite{Jacques2004} used beautiful techniques to
compute Betti numbers of edge ideals of several classes of
graphs. Edge ideals can be considered as path ideals of length $2$.
Our paper extends Jacques's techniques to higher
dimensions. Essentially, we consider the path ideal of a graph as a
disjoint union of connected components. We then use homological
methods to glue these components back together, and using Hochster's
formula we compute all top-degree graded Betti numbers, projective
dimension and regularity of path ideals of cycles.


The paper is organized as follows. In Section 2 we recall some
notation and basic algebraic and combinatorial concepts used in other
next chapters. In Section 3 we study the connected components of path
ideals which will provide us with the key to our homological
computations later in Section 4. Section 5 is where we apply the
homological results of Section 4 along to give a criterion to determine all non zero Betti numbers and projective dimension of path ideals
of cycles.
 While working on this paper the computer algebra systems
 CoCoA~\cite{CoCoA} and Macaulay2~\cite{Macaulay2} were used to test
 examples. We acknowledge the immense help that they have provided us
 in this project.


\section{Preliminaries}


Throughout we assume that $K$ is a field and $R=K\left[x_1,\dots,x_n\right]$ is a polynomial ring in $n$ variables.

\subsection*{Simplicial complexes and monomial ideals}

\begin{defn} An abstract \textbf{simplicial complex} on vertex set
 $\XX=\{x_1,\dots,x_n\}$ is a collection $\D$
  of subsets of $\XX$ satisfying
  \renewcommand{\theenumi}{\roman{enumi}}
\begin{enumerate}
\item $\{x_i\}\in \D$ for all $i$,
\item $F \in \D , G \subset F \Longrightarrow G \in \D$.
\end{enumerate}
The elements of $\D$ are called \textbf{faces} of $\D$ and the
maximal faces under inclusion are called \textbf{facets} of
$\D$. We denote the simplicial complex $\D$ with facets
$F_1,\dots,F_s$ by $\langle F_1,\dots,F_s\rangle$. We call
$\{F_1,\dots,F_s\}$ the facet set of $\D$ and is
denoted by $F(\D)$. The vertex set of $\D$ is denoted by $\ver(\D)$.
\end{defn}

\begin{defn} A \textbf{subcollection} of a simplicial complex $\D$
with vertex set $\XX$ is a simplicial complex whose facet set is a
subset of the facet set of $\D$. For $\YY\subseteq\XX$, an
\textbf{induced subcollection} of $\D $ on $\YY$, denoted by
$\D_{\YY}$, is the simplicial complex whose vertex set is a subset of
$\YY$ and facet set is $$\{F\in F(\D) \st F \subseteq {\YY}\}.$$
\end{defn}
If $F$ is a face of $\D=\langle F_1,\dots,F_s\rangle$,
we define the \textbf{complement} of $F$ in $\D$ to be
\begin{eqnarray*}
F_\XX^c=\XX\setminus F & \mbox{and}& \D_\XX^c=\langle (F_1)^c_\XX,\dots,(F_s)^c_\XX \rangle.
\end{eqnarray*}
Also if $\XX\subsetneqq \ver (\D)$, then
$\D^{c}_\XX=(\D_\XX)^{c}_\XX$.
\begin{defn} Let $R=K\left[x_1,\dots,x_n\right]$ be a
polynomial ring over a field $K$, and  $I$ an ideal in $R$ minimally
generated by square-free monomials $m_1,\ldots,m_s$.  One can associate
two simplicial complexes to $I$.
\renewcommand{\theenumi}{\roman{enumi}}
\begin{enumerate}
\item The \textbf{Stanley-Reisner complex} $\D_I$ associated
  to $I$ has vertex set $V=\{x_i \st x_i \notin I\}$ and is defined as $$\D_I = \{\{x_{i_{1}},\dots,x_{i_{k}}\} \st i_1 < i_2 < \dots < i_k, x_{i_{1}}\dots x_{i_{k}}\notin I \}.$$
\item The \textbf{facet complex}
  $\D(I)$ associated to $I$ has vertex set
 $\{x_1,\dots,x_n\}$ and is defined as
$$\D(I)=\langle F_1,\ldots,F_s \rangle \mbox{ where }
  F_i=\{x_j \st x_j | m_i,\ 1\leq j\leq n\}, \ 1 \leq i \leq s.$$
\end{enumerate}
\end{defn}
Conversely to a simplicial complex $\D$ one can associate two
monomial ideals.
\begin{defn} Let $\D$  be a simplicial complex with vertex set
$x_1,\ldots,x_n$ and $R=K\left[x_1,\dots,x_n\right]$ be a polynomial
  ring over a field $K$.
\renewcommand{\theenumi}{\roman{enumi}}
\begin{enumerate}
\item The \textbf{Stanley-Reisner} ideal of $\D$ is defined
as $$I_{\D}=( \prod_{x \in F}x  \st  F \notin \D).$$
\item   The \textbf{facet ideal} of $\D$ is defined as
$$I(\D)=( \prod_{x \in F}x \st \ F \mbox{ is a facet of} \D ).$$
\end{enumerate}
\end{defn}
Note that there is a one-to-one correspondence between monomial ideals
and simplicial complexes via each of these methods.
\begin{defn} Let $\D$ be a simplicial complex with vertex set
$\XX$. The \textbf{Alexander Dual} $\D^* $ is defined to be the
  simplicial complex with faces
$$\D^{*}=\{F^{c}_\XX   \st F \mbox{ is not face of } \D \}.$$
\end{defn}
To prove some of the results in this paper we need the definition of
a cone and its properties.
\begin{defn} Let $\D_1$ and $\D_2$ be two simplicial complexes on
disjoint vertex sets V and W. The \textbf{join} $\D_1 * \D_2$
is the simplicial complex on $ V\bigsqcup W$ with faces $F \cup G$
where $F\in \D_1, G \in \D_2$. The \textbf{cone}
$Cn(\D)$ of $\D$ is the simplicial complex
$\omega*\D$, where $\omega$ is a new vertex.
\end{defn}
Given a simplicial complex $\D$, we denote by ${\CC }.(\D)$ the
reduced chain complex and by $\widetilde{H}_i(\D)=Z_i(\D) /B_i(\D)$
the $i$ - th reduced homology groups of $\D$ with coefficients in
the field $K$.
For the proof of the following fact, see for example
Villarreal~\cite{Villarreal2001}, Proposition 5.2.5.
\begin{prop}\label{prop:prop2.16} If $\D$ is a simplicial complex
we have $\widetilde{H}_i(Cn(\D))=0$ for all $i$.
\end{prop}
\subsection*{Betti numbers}
 For any homogeneous ideal $I$ of the polynomial ring
 $R=K\left[x_1,\dots,x_n\right]$ there exists a \textbf{graded minimal
   finite free resolution}
$$0\rightarrow {\displaystyle
  \bigoplus_{d}}R(-d)^{\beta_{p,d}}\rightarrow\cdots{\displaystyle
   \rightarrow\bigoplus_{d}}R(-d)^{\beta_{1,d}}\rightarrow
 R\rightarrow R /I \rightarrow 0$$ of $R /I $ in which
 $R(-d)$ denotes the graded free module obtained by
 shifting the degrees of elements in $R$ by $d$. The numbers
 $\beta_{i,d}$, which we shall refer to as the $i$-th
$\mathbb{N}$-\textbf{graded Betti numbers} of degree $d$ of $R /I$,
 are independent of the choice of graded minimal finite free
 resolution.

For computing the $\mathbb{N}$-graded Betti numbers of the
Stanley-Reisner ring of a simplicial complex we use an equivalent form
of Hochster's formula.
\begin{theorem}\label{col:col15} Let $R=K[x_1,\dots,x_n]$ be a polynomial
ring over a field $K$, and $I$ be a pure square-free monomial ideal in
$R$. Then the $\mathbb{N}$-graded Betti numbers of $R/I$ are given
by $$\beta_{i,d}(R/I)={\displaystyle \sum_{\Gamma \subset \D(I),
   |\ver(\Gamma)|=d}} \hspace{.1 in} \dim_{K}\widetilde{H}_{i-2}
(\Gamma^c_{\ver (\Gamma)})$$ where the sum is taken over the induced subcollections $\Gamma$ of $\D(I)$ which have $d$ vertices.
\end{theorem}
    \begin{proof} Hochster's formula (see for example Corollary~5.12
      of~\cite{Miller2005}) says
        $$\beta_{i,d}(R/I)=\sum_{W \subset \ver(\D_I),|W|=d}
       \dim_{K}\widetilde{H}_{d-i-1}(\{\D_I\}_{W})$$ where
       $\{\D_I\}_{W}=\{F\in \D_I \st F\subset W\}$. On the
       other hand from \cite{Bruns1993} Lemma~5.5.3 we have
    \begin{equation}
      \widetilde{H}_{d-i-1}(\{\D_I\}_{W})\cong
      \widetilde{H}^{i-2}({\{\D_I\}_{W}}^{*})\cong
      \widetilde{H}_{i-2}({\{\D_I\}_{W}}^{*}).
      \label{eqn:orginalhochster}
     \end{equation}

     Suppose $m_1,m_2,\dots ,m_r$ is a minimal monomial generating set
     for $I$ and correspondingly, $\D(I)=\langle
    F_1,\ldots,F_s\rangle$. We now claim
     ${\{\D_I\}_{W}}^{*}=\left(\D(I)_{W}\right)_W^{c}$ for $W\subset
     \ver(\D_I)$.
   $$\begin{array}{ll}
     F\in {\{\D_I\}_W}^{*} &\Longleftrightarrow W\setminus F
      \notin \{\D_I\}_{W}\\
      & \Longleftrightarrow  W\backslash F\notin \D_I\\
     & \Longleftrightarrow \displaystyle\prod_{x\in W\setminus F}x \in
        I=(m_1,m_2,\dots,m_r)\\
     &\Longleftrightarrow \displaystyle  m_s| \prod_{x\in W\setminus F}x,
       \mbox{ for some } s \in \{1,\ldots,r\} \\
    &\Longleftrightarrow F_{s}\subset W\backslash F\subset W,
       \mbox{ for some } s \in \{1,\ldots,r\} \\
     &\Longleftrightarrow F\subset W\backslash F_{s}\in (\D(I)_{W})_W^{c},
        \mbox{ for some } s \in \{1,\ldots,r\}.
\end{array}$$

    Now Hochster's formula and (\ref{eqn:orginalhochster}) imply that
     \begin{equation*}
       \beta_{i,d}(R/I)=\displaystyle\sum_{W\subset
         \ver(\D),|W|=d}dim_{K}\widetilde{H}_{i-2}((\D(I)_{W})_W^{c}).
     \end{equation*}
   If we assume $\ver(\D(I)_{W})\neq W$ then clearly we have
   $(\D(I)_{W})_W^c$ is a cone and by Proposition~\ref{prop:prop2.16}
   it contributes zero to the sum. So we have
   $$\beta_{i,d}(R/I)=\sum_{\Gamma\subset \D(I),|\ver(\Gamma)|=d}
   dim_{K}\widetilde{H}_{i-2}(\Gamma_{\ver(\Gamma)}^{c}).$$ where the
   sum is taken over the induced subcollections $\Gamma$ of $\D(I)$
  which have $d$ vertices.
\end{proof}

Based on Theorem~\ref{col:col15}, from here on all induced
subcollections $\Gamma=\D_\YY$ of a simplicial complex $\D$ that we
consider will have the property that $\YY=\ver (\Gamma)$.

\section{Path ideals and runs}

We now focus on path ideals, path complexes, and their structures.
\begin{defn}
Let $G=(\XX,E)$ be a finite simple graph and $t$ be an integer such
that $t\geq 2$.  If $x$ and $y$ are two vertices of $G$, a
\textbf{path} of length $(t-1)$ from $x$ to $y$ is a
sequence of vertices $x=x_{i_1},\dots, x_{i_{t}}=y$ of $G$ such that
$\{x_{i_{j}},x_{i_{j+1}}\}\in E$ for all $j= 1,2,\dots,t-1$.
We define the {\bf path ideal} of $G$, denoted by $I_t(G)$ to be the
ideal of $K[x_1,\dots,x_n]$ generated by the monomials of the form
$x_{i_1}x_{i_2}\dots x_{i_t}$ where $x_{i_1},x_{i_2},\dots,x_{i_t}$ is
a path in $G$.
The facet complex of $I_t(G)$, denoted by $\D_t(G)$, is called the
{\bf path complex} of the graph $G$.
\end{defn}
Two special cases that we will be considering in this paper are when
$G$ is a {\bf cycle} $C_n$, or a {\bf line graph} $L_n$ on vertices
$\{x_1,\dots,x_n\}$.
$$C_n=\langle x_1x_2,\ldots,x_{n-1}x_n,x_nx_1\rangle \mbox{\ and \ }
L_n=\langle x_1x_2,\ldots,x_{n-1}x_n\rangle.$$
\begin{example} Consider the cycle $C_7$ with vertex set
 $\XX=\{x_1,\dots,x_7\}$
\begin{center}
\end{center}
Then we have
\begin{align*}
&I_4(C_7)=(x_1x_2x_3x_4,x_2x_3x_4x_5,x_3x_4x_5x_6,x_4x_5x_6x_7,x_1x_5x_6x_7,x_1x_2x_6x_7,x_1x_2x_3x_7)\\
&\D_4 (C_7)=\langle\{x_1,x_2,x_3,x_4\},\{x_2,x_3,x_4,x_5\},\{x_3,x_4,x_5,x_6\},\{x_4,x_5,x_6,x_7\},\{x_1,x_5,x_6,x_7\},\\
&\{x_1,x_2,x_6,x_7\},\{x_1,x_2,x_3,x_7\}\rangle.
\end{align*}

\end{example}

\begin{notation} Let $i$ and $n$ be two positive integers.
For (a set of) labeled objects we use the notation $\mod n$ to denote
$$x_i \mod n \ =\{x_j \st 1\leq j \leq n, i\equiv j \mod n\}$$
and
$$\{x_{u_1},x_{u_2},\dots,x_{u_t}\} \mod n\  =\{x_{u_j}\mod n \st j=1,2,\dots,n\}.$$
\end{notation}
\begin{note}\label{not:not1} Let $C_n$ be a cycle on vertex set
$\XX=\{x_1,\dots,x_n\}$ and $t< n$. The facets of the path complex
  $\D_t(C_n)=\langle F_1,\dots,F_n\rangle$ can be labeled as
  $$F_1=\{x_1,\dots,x_t\},\dots,F_{n-(t-1)}=\{x_{n-(t-1)},\dots,x_n\},\dots,F_{n}=\{x_1,\dots,x_{t-1},x_{n}\}$$
  such that $F_i=\{x_i,x_{i+1},\dots,x_{i+t-1}\}\mod n$ for all $1\leq
 i \leq n$.  This labeling is called the {\bf standard labeling} of
  $\D_t(C_n)$.

Since for each $1\leq i \leq n$ we have $$\begin{array}{llll}
   F_{i+1}\setminus F_{i}=\{x_{t+i}\}&\mbox{and}& F_{i}\setminus
   F_{i+1}=\{x_{i}\}&\mod n,
\end{array}$$
it follows that $\begin{array}{lllll} \left|F_i\setminus
  F_{i+1}\right|=1&\mbox{and}& \left|F_{i+1}\setminus
  F_{i}\right|=1&\mod n&\mbox{for all $1\leq i\leq
    n-1$}.
\end{array}$
\end{note}

It is clear that each induced subgraph of a graph-cycle is a disjoint union
of paths.  Borrowing the terminology from
S. Jacques~\cite{Jacques2004}, we call the path complex of a line a
``run'', and show that every induced subcollection of the path complex
of a cycle is a disjoint union of runs.
\begin{defn}\label{defn:defn3.5} Given an integer $t$, we define a
{\bf run} to be the path complex of a line graph. A run which has $p$
facets is called a \textbf{run of length $p$} and corresponds to
$\D_t(L_{p+t-1})$. Therefore a run of length $p$ has $p+t-1$ vertices.
\end{defn}

Proposition~\ref{prop:prop3.3} below shows that every proper induced
subcollection  of a path complex is a disjoint union of runs.

\begin{prop}\label{prop:prop3.3} Let $C_n$ be a cycle with
 vertex set $\XX=\{x_1,\dots.x_n\}$ and $2\leq t < n$. Let $\Gamma$ be
a proper induced connected subcollection of $\D_t(C_n)$ on
 $U\subsetneqq \XX$. Then $\Gamma$ is of the form
 $\D_t(L_{|\Gamma|})$, where $L_{|\Gamma|}$ is the line graph on
 $|\Gamma|$ vertices.
\end{prop}

   \begin{proof} Suppose $\D_t(C_n)=\langle F_1,\dots,F_n\rangle$ has
     standard labeling and $\Gamma=\langle
     F_{i_1},\dots,F_{i_r}\rangle$.  It is clear that there exists the
     facet $F_{a}\in \Gamma$ for $1\leq a \leq n$ such that
     $F_{a+1}\notin \Gamma \mod n$, because otherwise
     $\Gamma=\D_t(C_n)$. Therefore from Note~\ref{not:not1} we have
      \begin{eqnarray}
        \{x_a,x_{a+1},\dots,x_{a+t-1}\}\subset U&\mbox{and}&
         x_{a+t}\notin U. \label{eqn:U}
      \end{eqnarray}
Let $r$ be the largest non-negative integer such that
     $x_{a-i}\in U \mod n$ for $0\leq i \leq r$ so that
    \begin{eqnarray}
   \stackrel{\underbrace{x_{a-r-1},}}{\notin U}\
      \stackrel{\underbrace{x_{a-r},\ldots,x_{a-1},x_a,x_{a+1},\ldots,x_{a+t-1},}}{\in U}\
       \stackrel{\underbrace{\ x_{a+t}}.}{\notin U} \label{eqn:UU}
\end{eqnarray}

   It follows that since $\Gamma$ is an induced subcollection of
    $\D_t(C_n)$ on $U$, $F_{a-r},F_{a-r+1},\dots,F_a\in \Gamma$.
    We now show that $$F_i\notin \Gamma \mbox{ for all } i\notin
    \{a-r,a-r+1,\dots,a\} \mod n.$$ This follows from the fact that
    $\Gamma$ is connected: if any $F_i$ (except for $a-r\leq i \leq a
    \mod n$) intersects some of the facets $F_{a-r},\ldots,F_a \mod
    n$, then it must contain $x_{a-r-1}$ or $x_{a+t}$ (as otherwise it
    would be one of $F_{a-r},\ldots,F_a \mod n$), and hence $F_i
   \notin \Gamma$.

   We have therefore shown that $$ \Gamma=\langle
   F_{a-r},F_{a-r+1},\dots,F_a\rangle \mod n.$$
   Next we prove $\Gamma=\D_t(L_{|U|})$. Without loss of
   generality we can assume that $a-r=1$, so that $\Gamma=\langle
   F_1,\dots,F_{r+1}\rangle$. Since $\Gamma\neq \D_t(C_n)$ we can say
   that $r+1<n$ and therefore we have $$\ver
   (\Gamma)=\{x_1,x_2,\dots,x_{r+t}\}.$$ Since $\Gamma$ is induced and
   proper we have $r+t<n$, and therefore we can conclude
   that $$\Gamma=\D_t(L_{\{x_1,x_2,\dots,x_{t+r}\}}).$$
  \end{proof}

\begin{example} Consider the cycle $C_7$ on vertex set
$\XX=\{x_1,\dots x_7\}$ and the simplicial complex $\D_4(C_7)$. The
  following induced subcollections are two runs in
  $\D_4(C_7)$ $$\begin{array}{lll} \D_1&=&\langle
    \{x_1,x_2,x_3,x_4\},\{x_2,x_3,x_4,x_5\}\rangle\\ \D_2&=&\langle
    \{x_1,x_2,x_6,x_7\},\{x_1,x_2,x_3,x_7\},\{x_1,x_2,x_3,x_4\}\rangle.
\end{array}$$
\end{example}

\idiot{
\begin{lem}\label{col:recentcol} Let $\Gamma$ and $\Lambda$ be
two induced subcollections of $\D_t(C_n)=\langle
F_1,F_2,\dots,F_n\rangle$ which are both composed of runs of lengths
$s_1,\dots,s_r$. Then $\Gamma$ and $\Lambda$ are homeomorphic as
simplicial complexes. In particular the two simplicial complexes
$\Gamma^c$ and $\Lambda^c$ are homeomorphic and have the same reduced
homologies.
\end{lem}
      \begin{proof} First we suppose $\D_t(C_n)=\langle F_1,F_2,
   \dots,F_n\rangle$ has standard labeling. If we denote each run of
   length $s_i$ in $\Gamma$ and $\Lambda$ by $R_i$ and
   $R^{\prime}_{i}$, respectively, we have
  $$\Gamma=\langle R_1,R_2,\dots,R_r\rangle \mbox{ and }
  \Lambda=\langle {R_1}^{\prime},{R_2}^{\prime},\dots,{R_r}^{\prime}\rangle$$
   where, using the standard labeling, for $ 1\leq j_i,h_i \leq n$  we have
    $$R_i=\langle F_{j_i},F_{j_i+1},\dots,F_{j_i+s_i-1}\rangle \mbox{ and }
     {R_i}^{\prime}=\langle F_{h_i},F_{h_i+1},\dots,F_{h_i+s_i-1}\rangle
     \mod n.$$
   Then clearly $$\ver(\Gamma)=\displaystyle\bigcup_{i=1}^{r}\ver
   (R_i)=\bigcup_{i=1}^{r}\{x_{j_i},x_{j_i+1},\dots,x_{j_i+s_i+t-2}\}$$
   and $$\ver (\Lambda)=\displaystyle\bigcup_{i=1}^{r}\ver
   ({R_i}^{\prime})=\bigcup_{i=1}^{r}\{x_{h_i},x_{h_i+1},\dots,x_{h_i+s_i+t-2}\}.$$
   Now we define the function $\varphi:\ver (\Gamma)\longrightarrow
   \ver (\Lambda)$ where $$\varphi(x_{j_i+u})=x_{h_i+u} \mbox{ for }
   0\leq u \leq s_i+t-2.$$
   Since $\varphi$ is a bijective map between vertex set $\Gamma$ and
   $\Lambda$ which preserves faces, we can conclude $\Gamma$ and
   $\Lambda$ are homeomorphic. Therefore, two simplicial complexes
   $\Gamma^c$ and $\Lambda^c$ are homeomorphic as well and have the
   same reduced homology.
   \end{proof}
}

If $\Gamma$ and $\Lambda$ are two induced subcollections of
$\D_t(C_n)$ composed of runs of of equal lengths, using a bijective
map between their vertex sets one can easily see that $\Gamma$ and
$\Lambda$ are homeomorphic as simplicial complexes. In particular the
two simplicial complexes $\Gamma^c$ and $\Lambda^c$ are homeomorphic
and have the same reduced homologies. Therefore, in light of
Proposition~\ref{prop:prop3.3} and Theorem~\ref{col:col15} all the
information we need to compute the Betti numbers of $\D_t(C_n)$, or
equivalently the homologies of induced subcollections of $\D_t(C_n)$,
depend on the number and the lengths of the runs.

\begin{defn}\label{defn:e(s_1)} For a fixed integer $t \geq 2$,
let the pure $(t-1)$-dimensional simplicial complex $\Gamma=\langle
F_1,\ldots,F_s\rangle$ be a disjoint union of runs of length
$s_1,\ldots,s_r$. Then the sequence of positive integers
$s_1,\ldots,s_r$ is called a \textbf{run sequence} on
$\YY=\ver(\Gamma)$, and we use the notation
$$E(s_1,\dots,s_r)=\Gamma^c_\YY=\langle(F_1)_{{\YY}}^{c},\dots,
(F_s)_{\YY}^{c}\rangle.$$
\end{defn}

\section{Reduced homologies for Betti numbers}

 Let $I=I_t(C_n)$ be the path ideal of the the cycle $C_n$ for some $t
 \geq 2$.  By applying Hochster's formula (Theorem~\ref{col:col15}),
 we see that to compute the Betti numbers of $R/I$, we need to compute
 the reduced homologies of complements of induced subcollections of
 $\D$ which by Proposition~\ref{prop:prop3.3} are disjoint unions of
 runs. This section is devoted to complex homological
 calculations. The results here will allow us to compute all Betti
 numbers of $R/I$ (and more) in the sections that follow.  We begin by
 recalling the {Mayer-Vietoris} sequence; for more details see for
 example Hatcher~\cite {Hatcher2002} Chapter 2.

Suppose $\D$ is a simplicial complex and $\D_1$ and $\D_2$ are two
subcollections such that $\D=\D_1\cup\D_2$. We have the exact
sequence of the chain complexes $$0\rightarrow
{\CC}.(\D_1\cap\D_2)\rightarrow
{\CC}.(\D_1)\oplus{\CC}.(\D_2)\rightarrow {\CC}.(\D)\rightarrow 0$$
which produces the following long exact sequence,
called the \textbf{ Mayer-Vietoris} sequence $$\dots \rightarrow
\widetilde{H}_i(\D_1\cap\D_2)\rightarrow
\widetilde{H}_i(\D_1)\oplus\widetilde{H}_i(\D_2)\rightarrow
\widetilde{H}_i(\D)\rightarrow
\widetilde{H}_{i-1}(\D_1\cap\D_2)\rightarrow \dots .$$
We make a basic observation.
\begin{lem}\label{lem:lem2} Let $E_1,\dots,E_m$ be subsets of the finite
set $V$ where $s\geq 2$ and suppose that ${\EE}=\langle (E_1)_V^c,
(E_2)_V^c, ..., (E_m)_V^c \rangle$. Then for any $i$ we have
\renewcommand{\theenumi}{\roman{enumi}}
\begin{enumerate}
\item Suppose ${\displaystyle V\setminus\bigcup_{j=2}^s
  E_j}\neq \emptyset$. If ${\EE_1}=\langle (E_1)_V^{c}\rangle$ and
  $\EE_2=\langle (E_2)_V^{c},...,(E_m)_V^{c}\rangle$ then
\begin{align*}
\widetilde{H}_i({\EE})=\widetilde{H}_i({\EE_1}\cup {\EE_2}) & \cong\widetilde{H}_{i-1}({\EE_1}\cap {\EE_2})~\vspace{.1 in}\\
&=\widetilde{H}_{i-1}(\langle (E_1\cup E_2)_V^{c},...,(E_1 \cup E_m)_V^{c}\rangle)\\
&=\widetilde{H}_{i-1}(\langle (E_2)_{(V\setminus E_1)}^{c},...,(E_m)_{(V \setminus E_1)}^{c}\rangle).
\end{align*}
\item If $E_a \subset E_b$ for some $a\neq b$, then $\EE=\langle
  (E_1)_v^c,\ldots,(\widehat{E_b})_V^c,\ldots,(E_m)_V^c \rangle$.
\end{enumerate}
\end{lem}
The decomposition ${\EE}={\EE_1}\cup{\EE_2}$ described above is called
{\bf standard decomposition} of ${\EE}$.
      \begin{proof} The proof of (ii) is trivial so we shall only
        prove (i).
        Since ${\EE_1}$ is a simplex, we have
        $\widetilde{H}_i({\EE_1})=0$. Also since ${\displaystyle
          V\setminus\bigcup_{i=2}^s E_i}\neq \emptyset$ we have
        ${\EE_2}$ is a cone, so from Proposition~\ref{prop:prop2.16}
        we conclude $$\widetilde{H}_{i}({\EE}_2)=0\mbox{ for all }i.$$
        Now ${\EE_1}\cap{\EE_2}=\langle {(E_1\cup E_2)}_V^{c},...,{(E_1
          \cup E_m)}_V^{c}\rangle$, so by applying the Mayer-Vietoris
        sequence we reach the following exact
        sequence
       $$\cdots \longrightarrow \widetilde{H}_i({\EE_1}) \oplus
        \widetilde{H}_i({\EE_2})\longrightarrow
        \widetilde{H}_i({\EE})\longrightarrow
        \widetilde{H}_{i-1}({\EE_1}\cap {\EE_2})\longrightarrow
        \widetilde{H}_{i-1}({\EE_1})\oplus\widetilde{H}_{i-1}({\EE_2})
        \longrightarrow \cdots $$ which implies
        that
       \begin{align*}
          \widetilde{H}_i({\EE_1}\cup
          {\EE_2})\cong&\widetilde{H}_{i-1}({\EE_1}\cap
          {\EE_2})~\vspace{.1 in}\\ =&\widetilde{H}_{i-1}(\langle
          {(E_1\cup E_2)}_V^{c},...,{(E_1 \cup E_m)}_V^{c}\rangle).
       \end{align*}
\end{proof}
\begin{prop}\label{prop:prop9} Let $\Gamma=\langle
  E_1,\dots,E_m\rangle$ be a pure simplicial complex of dimension
  $t-1$ over the vertex set $V=\{x_1,\dots,x_n\}$ where $2 \leq t \leq
  n$.  Suppose the connected components of $\Gamma$ are runs of
  lengths $s_1,\dots,s_r$, and ${\EE}=E(s_1,\dots,s_r)$. Let
 $s_j=(t+1)p_j+d_j$ where $p_j \geq 0$ and $0\leq d_j\leq t$ and
  $1\leq j \leq r$. Then for all $i$, we have
\bigskip

\begin{tabular}{rll}
$i.$ & If $s_j\geq t+2$ then&
$\widetilde{H}_i({\EE}))\cong \widetilde{H}_{i-2}(E(s_1,\dots, s_j-(t+1),\dots,s_r));$\\
&&\\
$ii.$ & If $d_j \neq 1,2$ then&
$\widetilde{H}_{i}(\EE)=0$;\\
&&\\
$iii.$&If $s_j=2$ and $r\geq 2$ then&
$\widetilde{H}_{i}({\EE})=\widetilde{H}_{i-2}(E(s_1,\dots,s_{j-1},s_{j+1},\dots,s_r))$;\\
&&\\
$iv.$& If $s_j=1$ and $r\geq 2$ then &
$\widetilde{H}_{i}({\EE})=\widetilde{H}_{i-1}(E(s_1,\dots,s_{j-1},s_{j+1},\dots,s_r))$.
\end{tabular}
\end{prop}
     \begin{proof} We assume without loss of generality that $E_1,\dots,E_m$
    are ordered such that $E_1,\dots,E_{s_j}$ are the facets of the
       run of length $s_j$, and they have standard labeling
     $$E_1=\{x_1,x_2,\dots,x_t\},E_2=\{x_2,x_3,\dots,x_{t+1}\},\dots,
         E_{s_j}=\{x_{s_j},x_{s_j+1},\dots,x_{s_j+t-1}\}\mod n.$$
    We have $\EE=\langle (E_1)_{V}^{c},(E_2)_{V}^{c},...,(E_m)_{V}^{c}
    \rangle$. Since $x_1\in V\setminus \bigcup_{i=2}^m E_i$ there is a
    standard decomposition $$\EE=\langle
   (E_1)_{V}^{c}\rangle\cup\langle
    (E_2)_{V}^{c},\ldots, (E_m)_{V}^{c}\rangle.$$ From
    Lemma~\ref{lem:lem2} (i), setting $V^{\prime}=V \setminus
    \{x_1,x_2,\dots,x_t\}$, we have
  \begin{eqnarray} \widetilde{H}_i(\EE) \cong \widetilde{H}_{i-1}
 (\langle ( E_2)_{V'}^{c},\ldots,(E_m)_{V'}^{c}\rangle).
   \label{eqn:newedition}
    \end{eqnarray}
 If $s_j\geq t+2$ from (\ref{eqn:newedition}) we have
  \begin{eqnarray} \widetilde{H}_i(\EE))=\widetilde{H}_{i-1}(\langle
     {\{x_{t+1}\}}_{V^{\prime}}^{c},(E_{t+2})_{V^{\prime}}^{c},\dots,
     (E_m)_{V^{\prime}}^{c}\rangle) \label{eqn:edition}
    \end{eqnarray}
  and since the following is a standard decomposition
   $$\langle
 {\{x_{t+1}\}}_{V^{\prime}}^{c}\rangle\cup\langle(E_{t+2})_{V^{\prime}}^{c},
\dots,(E_{s_j})_{V^{\prime}}^{c},\dots,(E_m)_{V^{\prime}}^{c}\rangle$$
from (\ref{eqn:edition}), Lemma~\ref{lem:lem2} (i) and by setting
  $V^{\prime\prime}=V \setminus \{x_1,x_2,\dots,x_{t+1}\}$, we have
  $$\widetilde{H}_i(\EE)\cong \widetilde{H}_{i-2}(\langle
(E_{t+2})_{V^{\prime\prime}}^c,\dots,
  (E_{s_j})_{V^{\prime\prime}}^c,
  \dots,(E_m)_{V^{\prime\prime}}^c\rangle).$$ Now the connected
  components of $\langle E_{t+2},\ldots,E_m\rangle$ are runs of
  lengths $s_1,\ldots,s_j-(t+1),\dots,s_r$, and therefore we can
  conclude that for all $i$
   $$\widetilde{H}_i(\EE)=\widetilde{H}_{i-2}(E(s_1,\dots,s_j-(t+1),
   \dots, s_r)).$$
This settles Case (i) of the proposition. Now suppose $1 \leq s_j <
t+2$. In this case by (\ref{eqn:newedition}) and Lemma~\ref{lem:lem2}
(i) and (ii) we see that
\begin{eqnarray} \widetilde{H}_i(\EE)\cong\widetilde{H}_{i-1}(\langle {\{x_{t+1}\}}_{V^{\prime}}^{c},(E_{s_j+1})_{V^{\prime}}^{c},\dots, (E_m)_{V^{\prime}}^{c}\rangle) &\mbox{for all $i$}. \label{eqn:edition2}
\end{eqnarray}
\begin{enumerate}
\item If $ s_j\geq 3$ since $x_{s_j+t-1}\in V^{\prime}\setminus
  (\bigcup_{i=s_j+1}^m E_i\cup \{x_{t+1}\})$ the simplicial
  complex $$\langle {\{x_{t+1}\}}_{V^{\prime}}^{c},
  (E_{s_j+1})_{V^{\prime}}^{c},\dots,(E_m)_{V^{\prime}}^{c}\rangle$$
  is a cone and by Proposition~\ref{prop:prop2.16} and
  (\ref{eqn:edition2}) we have $\widetilde{H}_i(\EE) =0$ for all $i$.
\item If $s_j=2$ and $r\geq 2$, since $x_{t+1}\in V^{\prime}\setminus
  (\bigcup_{i=s_j+1}^m E_i)$ we have $$\langle
  {\{x_{t+1}\}}_{V^{\prime}}^{c}\rangle\cup\langle(E_{s_j+1})_{V^{\prime}}^{c},
  \dots,(E_m)_{V^{\prime}}^{c}\rangle$$ is a standard decomposition
  and then by Lemma~\ref{lem:lem2} and (\ref{eqn:edition2}) we
  have $$\begin{array}{llll}
    \widetilde{H}_i(\EE)&\cong&\widetilde{H}_{i-2}(\langle
    (E_{s_j+1})_{V^{\prime\prime}}^{c},\dots,(
    E_m)_{V^{\prime\prime}}^{c}\rangle)~\vspace{.1 in}
    \\ &=&\widetilde{H}_{i-2}(E(s_1,\dots,s_{j-1},s_{j+1},\dots,s_r))&\mbox{for
      all $i$.}
\end{array}$$
This settles Case~(iii).
\item If $s_j=1$ and $r\geq 2$ since $E_1\cap E_{h}=\emptyset$ for $ 1
  < h \leq s$, and from (\ref{eqn:newedition}) we
 have $$\widetilde{H}_i(\EE)\cong \widetilde{H}_{i-1}(E(s_1,\dots
    s_{j-1},s_{j+1},\dots,s_r)) \mbox{ for all } i.$$
This settles Case (iv).
\end{enumerate}
To prove (ii), we use induction on $p_j$. If $p_j=0$, then
$d_j=s_j\geq 1$. From above we know that $\widetilde{H}_i(\EE)=0$ if
$3\leq s_j\leq t$, and we are done.  Now suppose $p_j \geq 1$ and the
statement holds for all values less than $p_j$. We have two cases:
\begin{enumerate}
\item If $s_j<t+2$, then since $p_j\geq 1$, we must have $p_j=1$,
  $d_j=0$, and $s_j=t+1$. It was proved above (under the case $3 \leq
  s_j < t+2$) that $\widetilde{H}_i(\EE)=0$.
\item If $s_j\geq t+2$, by (i) we have
$$\widetilde{H}_i(\EE)\cong\widetilde{H}_{i-2}(E(s_1,\dots,(t+1)(p_j-1)+d_j,
  \dots,s_r))=0 \mbox{  when }d_j\neq 1,2.$$
\end{enumerate}
This proves (ii) and we are done.
   \end{proof}
We conclude that for computing the homology of the induced
subcollections of path complexes of  cycles or lines the only cases
which have to be considered are those of runs of length one or two. We
now set about computing these.
\begin{prop}\label{prop:prop13} Let $t$ and $n$ be integers, such that
$2\leq t \leq n$. Let $\alpha,\beta \geq 0$ and
  consider $${\EE}=E((t+1)p_1+1,\dots,(t+1)p_{\alpha}+1,(t+1)q_1+2,\dots,(t+1)q_{\beta}+2)$$
  for nonnegative integers
  $p_1\dots,p_{\alpha},q_1,\dots,q_{\beta}$. Then
$$\widetilde{H}_i({\EE})=\left\{
\begin{array}{ll}
  K & i=2(P+Q)+2\beta+\alpha-2 \\
  0 & \mbox{otherwise}
\end{array}\right.$$
where $P=\sum_{i=1}^\alpha p_i$ and $Q=\sum_{i=1}^\beta q_i$.
\end{prop}
From here on, we use the notation $E(1^\alpha,2^\beta)$ to denote the
complex $\EE$ described in the statement of
Proposition~\ref{prop:prop13} in the case where all the $p$'s and
$q$'s are zero; i.e. the case of $\alpha$ runs of length one and
$\beta$ runs of length two.
\begin{proof}
First we prove the two cases $\alpha=0$, $\beta=1$ and $\alpha=1$, $\beta=0$.
\begin{enumerate}
\item If $\alpha=1$, $\beta=0$,  then ${\EE}\cong\langle V\setminus\{x_1,x_2,\dots,x_t\}\rangle=\{\emptyset\}$
where $V=\{x_1,\dots,x_t\}$, and therefore
$$\widetilde{H}_i ({\EE})=\left\{
\begin{array}{ll}
K & i=-1 \\
0 &  \mbox{otherwise.}
\end{array}\right.$$
\item If $\alpha=0$, $\beta=1$, then ${\EE}\cong\langle
  (\{x_1,x_2,\dots,x_{t}\})^{c}_{V},(\{x_2,\dots,x_{t+1}\})^{c}_{V}\rangle=\langle
  \{x_{t+1}\},\{x_1\}\rangle $ where $V=\{x_1,\dots,x_{t+1}\}$.
  Since ${\EE}$ is disconnected, we have
$$\widetilde{H}_i ({\EE})=\left\{
\begin{array}{ll}
  K &i=0 \\
 0 & \mbox{otherwise.}
\end{array}\right. $$
\end{enumerate}
To prove the statement of the proposition, we use repeated
applications of Proposition~\ref{prop:prop9}~(i), $p_1$ times to the
first run, $p_2$ times to the second run, and so on till $q_\beta$
times to the last run as follows.
 \begin{align*}
    \widetilde{H}_i({\EE})=&\ \widetilde{H}_i
    (E((t+1)p_1+1,\dots,(t+1)p_{\alpha}+1,(t+1)q_1+2,\dots,(t+1)q_{\beta}+2))\\
      \cong &\ \widetilde{H}_{i-2}
   (E((t+1)(p_1-1)+1,\dots,(t+1)p_{\alpha}+1,(t+1)q_1+2,\dots,(t+1)q_{\beta}+2))\\
      \vdots&\\
\cong&\ \widetilde{H}_{i-2(P+Q)}(E(1^{\alpha},2^{\beta})) \hspace{1in} \mbox{apply Proposition~\ref{prop:prop9} (iv)}\\
\cong &\ \widetilde{H}_{i-2(P+Q)-\alpha}(E(2^\beta))  \hspace{1.05in}\mbox{apply Proposition~\ref{prop:prop9} (iii)}\\
\cong &\ \widetilde{H}_{i-2(P+Q)-\alpha-2\beta+2}(E(2)) \hspace{.75in} \mbox{apply Case 2 above}\\
 =&\left\{\begin{array}{ll}
              K &   i=2(P+Q)+\alpha+2\beta-2 \\
              0 &\mbox{otherwise.}
            \end{array}\right.
\end{align*}
\end{proof}
An immediate consequence of the above calculations is the homology of the complement of a run, or equivalently, the path complex of any line graph.
\begin{col}\label{col:one-run} Let $t$, $p$ and $d$ be integers
such that $t \geq 2$, $p\geq 0$,and $0\leq d\leq t$.  Then
$$\widetilde{H}_i({E((t+1)p+d)})=\left\{
\begin{array}{ll}
  K & d=1,\ i=2p-1 \\
  K & d=2,\  i=2p \\
  0 & \mbox{otherwise.}
\end{array}\right.$$
\end{col}
  \begin{proof} By Proposition~\ref{prop:prop9}~(ii), if $d\neq 1,2$
  the homology is zero. In the cases where $d=1,2$, the result follows
 directly from Proposition~\ref{prop:prop13}.
\end{proof}
We end this section with the calculation of the homology of the
complement of the whole path complex of a cycle; this will give us the
top degree Betti numbers of the path ideal of a cycle. We will first
need a technical lemma.
\begin{lem}\label{lem:lem3} Let $R=K\left[x_1,\dots,x_n\right]$
be a polynomial ring over a field $K$, and suppose $\D_t(C_n)=\langle
F_1,F_2,\dots,F_n\rangle$ is the path complex of a cycle $C_n$ with
standard labeling.  Let $a,k,s,t \in \{1,\ldots,n\}$ be such that $
k<t$, and $a+s+t-1<n$. Suppose $s=(t+1)p+d$ where $p\geq 0$ and $0\leq
d < t+1$. Set $V=\{x_a,x_{a+1},\dots,x_{a+s+t-1}\}$
and $${\EE}=\langle
(F_{a})_{V}^{c},\dots,{(F_{a+s-1})}_{V}^{c},\{x_{a+s+t-k},x_{a+s+t-k+1},\dots,x_{a+s+t-1}\}_{V}^{c}
\rangle.$$
Then for all $i$ we have
$$\widetilde{H}_i({\EE})=\left\{
\begin{array}{ll}
K & d=1, i=2p\\
K & d=k+1, i=2p+1\\
0 & \mbox{otherwise.}
\end{array}
\right.$$
\end{lem}
   \begin{proof} Without loss of generality we can assume $a=1$ so that
  $ V=\{x_1,\dots,x_{s+t}\}$ and
     $${\EE}=\langle
    (F_{1})_{V}^{c},\dots,{(F_{s})}_{V}^{c},\{x_{s+t-k+1},\dots,x_{s+t}\}_{V}^{c}
    \rangle.$$
   Since $x_{s+t}\notin F_h$ for $1\leq h\leq s$, $\EE$ has standard
    decomposition
     $$\EE=\langle(F_1)_{V}^{c},(F_2)_{V}^{c},\dots,(F_{s})_{V}^{c}\rangle\cup \langle \{x_{s+t-k+1},x_{s+t-k+2},\dots,x_{s+t}\}_{V}^{c}\rangle$$
    and then from Lemma~\ref{lem:lem2} (i) and (ii), setting
    $V_1=V\setminus \{x_{s+t-k+1},x_{s+t-k+2},\dots,x_{s+t}\}$, we have
\begin{align}
\widetilde{H}_{i}(\EE)\cong& \widetilde{H}_{i-1}
\langle (F_1)_{V_1}^{c},(F_2)_{V_1}^{c},\dots,(F_{s-k})_{V_1}^{c},\{x_{s-k+1},\dots,x_{s+t-k}\}_{V_1}^{c},\{x_{s-k+2},\dots,x_{s+t-k}\}_{V_1}^{c},
\notag \\
&\hspace{.45in},\dots,{\{x_{s-1},\dots,x_{s+t-k}\}}_{V_1}^{c},{\{x_{s},\dots,x_{s+t-k}\}}_{V_1}^{c}\rangle) \notag\\
=&\widetilde{H}_{i-1}
(\langle (F_1)_{V_1}^{c},(F_2)_{V_1}^{c},\dots,(F_{s-k})_{V_1}^{c},{\{x_{s},\dots,x_{s+t-k}\}}_{V_1}^{c}\rangle ) \label{eqn:mylemma}
\end{align}
We prove our statement by induction on $|V|=s+t=(t+1)p+d+t$.
The base case is $|V|=d+t$, in which case $p=0$ and $d=s\geq 1$. There are
two cases to consider.
\begin{enumerate}
\item If $1 \leq d \leq k$, then $s \leq k$, and so by~(\ref{eqn:mylemma})
  $$\widetilde{H}_{i}(\EE)=\widetilde{H}_{i-1}( \langle \{x_{s},\dots,x_{s+t-k}\}_{V_1}^{c} \rangle ).$$
    The
  simplex ${\{x_{s},\dots,x_{s+t-k}\}}_{V_1}^{c}$ is not empty unless $s=d=1$, and hence we have
$$\widetilde{H}_i({\EE})=\left\{
\begin{array}{ll}
K & d=1, i=0 \\
0 & \mbox{otherwise.}
\end{array}
\right . $$
\item If $d > k$, we use~(\ref{eqn:mylemma}) to note that since
  $x_{s+t-k}\notin F_1 \cup \ldots\cup F_{s-k}$, the following is a
  standard decomposition$$\langle
  (F_1)_{V_1}^{c},(F_2)_{V_1}^{c},\dots,(F_{s-k})_{V_1}^{c},{\{x_{s},\dots,x_{s+t-k}\}}_{V_1}^{c}
  \rangle.$$
Using Lemma~\ref{lem:lem2} and~(\ref{eqn:mylemma}) along with the fact that $s=d \leq
t$, we find that if $V_2=V\setminus \{x_{s},\dots,x_{s+t}\}$, then
\begin{eqnarray*}
\widetilde{H}_{i}({\EE})&\cong&\widetilde{H}_{i-2}(\langle {\{x_{1},\dots,x_{s-1}\}}_{V_2}^{c},\{x_2,\ldots,x_{s-1}\}_{V_2}^c,\dots,{\{x_{s-k},\dots,x_{s-1}\}}_{V_2}^{c}\rangle)~\vspace{.1 in}\\
 &\cong&\widetilde{H}_{i-2}( \langle {\{x_{s-k},\dots,x_{s-1}\}}_{V_2}^{c}\rangle).
\end{eqnarray*}
Now the simplex $\{x_{s-k},\dots,x_{s-1}\}_{V_2}^c$ is nonempty
unless $s-k=1$, or in other words, $d=s=k+1$. Therefore
$$\widetilde{H}_i({\EE})=\left\{
\begin{array}{ll}
K &  d=k+1, i=1\\
0 & \mbox{otherwise.}
\end{array}\right . $$
\end{enumerate}
This settles the base case of the induction.
Now suppose $|V|=s+t>d+t$ and the theorem holds for all the cases
where $|V|<s+t$.  Since $|V_1|=(s-k)+t<|V|$ we shall apply
(\ref{eqn:mylemma}) and use the induction hypothesis on $V_1$, now
with the following parameters: $k_1=t-k+1$, $s_1=s-k=(t+1)p+d-k$ and
$$d_1=\left\{\begin{array}{ll}
d-k & d\geq k\\
d-k+t+1 & d<k
\end{array}
\right . \mbox{ \ \ \ and \ \ \ \  }
p_1=\left\{
\begin{array}{ll}
p & d\geq k\\
p-1 & d<k.
\end{array}
\right . $$
Applying the induction hypothesis on $V_1$ we see that
$\widetilde{H}_{i}({\EE})=0$ unless one of the following scenarios
happen, in which case $\widetilde{H}_{i}({\EE})=K$.
\begin{enumerate}
\item $d_1=1$ and $i-1=2p_1$.
  \begin{enumerate}
   \item When $d\geq k$, this means that $d=k+1$ and $i=2p+1$.
   \item When $d < k$, this means $d-k+t+1=1$ which implies that $0
     \leq d=k-t\leq 0$, and hence $d=0$ and $t=k$, which is not possible
      as we have assumed $k<t$.
  \end{enumerate}
\item $d_1=k_1+1$ and $i-1=2p_1+1$.
  \begin{enumerate}
   \item When $d\geq k$, this means that $d-k=t-k+1+1$ and so $d=t+2$
     which is not possible, as we have assumed $d<t+1$.
   \item When $d < k$, this means that $d-k+t+1=t-k+1+1$ and so
     $d=1$ and $i=2p_1+2=2p$.
  \end{enumerate}
\end{enumerate}
We conclude that $\widetilde{H}_{i}({\EE})=K$ only when $d=1$ and
$i=2p$, or $d=k+1$ and $i=2p+1$, and $\widetilde{H}_{i}({\EE})=0$
otherwise.
\end{proof}
\begin{lem}\label{lem:newlem1} 
Let $2\leq t \leq n$ and $\D=\D_t(C_n)$
be the path complex of a cycle $C_n$ with vertex set
$\XX=\{x_1,x_2,\dots, x_n\}$. Suppose $n=(t+1)p+d$ where $p\geq 1$,
$0\leq d \leq t$ and consider the following simplicial complexes
     \begin{eqnarray}
      \begin{array}{l}
       E_0=\langle
       (F_1)_{\XX}^{c},(F_2)_{\XX}^{c},\dots,(F_{n-t+1})_{\XX}^{c}\rangle=E(n-t+1)\\
       \\
       E_{k}=E_{k-1}\cup \langle
       (F_{n-k+1})_{\XX}^{c}\rangle\mbox{ for }k=1,2,\dots,t-1.
      \end{array}\label{eqn:204}
    \end{eqnarray}
Then for all $0 \leq k \leq t-2$, we have
  \begin{eqnarray*}
  \widetilde{H}_{i}(E_{k}\cap\langle{(F_{n-k})}_\XX^{c}\rangle)=\left\{
       \begin{array}{ll}
          K & d=0,\ i=2p-3\\
         K & d=t-k-1,\ i=2p-2  \\
           0 & \mbox{otherwise.}
        \end{array}\right.
 \end{eqnarray*}

\end{lem}
\begin{proof}
  Setting $\XX^{\prime}=\XX\setminus
  F_{n-k}=\{x_{t-k},\ldots,x_{n-k-1}\}$ we can write
\begin{align}
E_{k}\cap \langle {(F_{n-k})}_\XX^{c}\rangle= \langle
 (F_1)_{\XX'}^c,\ldots,(F_{n-t+1})_{\XX'}^c,(F_{n-k+1})_{\XX'}^c,
 \ldots,(F_n)_{\XX'}^c\rangle.\label{e:intersection}
  \end{align}
 We now compute the
 $(F_h)_{\XX'}^c=\{x_{h},\ldots,x_{h+(t-1)}\}_{\XX'}^c$ appearing in
\ref{e:intersection}).
 \begin{itemize}
  \item When $1 \leq h \leq t-k$ it is clear that
    $$(F_{h})_{\XX'}^c= \{x_{t-k},x_{t-k+1},\dots,x_{t+h-1}\}_{\XX'}^c.$$
 \item When  $t-k+1 \leq h \leq n-t-k-1$ then $2t-k\leq h+t-1 \leq n-k-2$,
  and so
    $$(F_{h})_{\XX'}^c=
    \{x_{t-k},\ldots,x_{h-1},x_{h+t},\ldots,x_{n-k-1}\}.$$
   \item When $n-k-t\leq h \leq n-t+1$, then $n-k-1\leq h+t-1 \leq n$,
     and therefore
  $$(F_{h})_{\XX'}^c = \{x_{h},\dots,x_{n-k-1}\}_{\XX'}^c.$$
 \item When $h=n-j$ for $0 \leq j \leq k-1$. Then $t-k\leq -j+(t-1)\leq
     t-1$ and so we have
    $$F_{n-j}=\{x_{n-j},\ldots,x_{n-j+(t-1)}\}= \{x_{n-j},
  \ldots,x_n,x_1,\ldots, x_{t-j-1}\} \mod n$$ which implies that
   $$(F_{n-j})_{\XX'}^c=\{x_{t-k},x_{t-k+1},\dots,x_{t-j-1}\}_{\XX'}^c.$$
 \end{itemize}
  From the observations above, (\ref{e:intersection}) and
  Lemma~\ref{lem:lem2}~(ii) we see that
  \begin{eqnarray}
     E_{k}\cap \langle {(F_{n-k})}_\XX^{c}\rangle= \langle
     \{x_{t-k}\}, F_{t-k+1},\dots, F_{n-t-k-1}, \{x_{n-t+1},
       \ldots,x_{n-k-1}\}\rangle^{c}_{\XX^{\prime}}.\label{eqn:two-first}
\end{eqnarray}
We now consider the  following scenarios.
\begin{enumerate}
   \item Suppose $p=1$.  In this situation, $n=t+d+1\leq 2t+1$ which
     implies that $n-t-k-1\leq t-k$. Therefore, (\ref{eqn:two-first})
      becomes
    \begin{eqnarray}
     E_{k}\cap \langle {(F_{n-k})}_\XX^{c}\rangle= \langle
    \{x_{t-k}\}, \{x_{n-t+1},
       \ldots,x_{n-k-1}\}\rangle^{c}_{\XX^{\prime}}.\label{eqn:two-third}
    \end{eqnarray}
      \begin{enumerate}
      \item If $d\leq t-k-2$, then $n-t+1=t+d+1-t+1=d+2\leq t-k$. As
        well, since $n\geq t+1$, we have $n-k-1\geq t-k$. It follows
       that in this situation, $x_{t-k}\in \{x_{n-t+1},
       \ldots,x_{n-k-1}\}$ which means that (\ref{eqn:two-third})
         becomes $E_{k}\cap \langle {(F_{n-k})}_\XX^{c}\rangle=
        \langle \{x_{t-k}\}^{c}_{\XX^{\prime}}\rangle.$
      Also note that $\XX'=\{x_{t-k}\}$ only when $d=0$.
     It follows that
    $$\widetilde{H}_i(E_{k}\cap \langle {(F_{n-k})}^{c}\rangle)
         \cong\left\{\begin{array}{ll}
             K & d=0,\ i=-1 \\
           0 & \mbox{otherwise.}
           \end{array}\right. $$
      \item If $d > t-k-2$. In this situation, $x_{t-k}\notin
        \{x_{n-t+1}, \ldots,x_{n-k-1}\}$ which means that we can apply
        Lemma~\ref{lem:lem2}~(i), with
        $\XX^{\prime\prime}=\XX^{\prime}\setminus \{x_{t-k}\}$ to find
       that for all $i$
        $$\widetilde{H}_{i}(E_{k}\cap\langle
        {(F_{n-k})}^{c}\rangle)=\widetilde{H}_{i-1}( \langle
        \{x_{n-t+1},\dots,x_{n-k-1}\}
        \rangle^{c}_{\XX^{\prime\prime}}).$$
   Moreover, $\XX''=\{x_{n-t+1}, \ldots,x_{n-k-1}\}$ only when
       $d=t-k-1$, and so we have
   $$\widetilde{H}_i(E_{k}\cap \langle {(F_{n-k})}^{c}\rangle)
         \cong\left\{\begin{array}{ll}
             K & d=t-k-1,\ i=0 \\
             0 & \mbox{otherwise.}
           \end{array}\right. $$
     \end{enumerate}

 \item Suppose $p\geq 2$.  In this case it is easy to see that
    $n-t-k-1>t-k$ and $n-t+1 > t-k$. Therefore, we can apply
   Lemma~\ref{lem:lem2}~(i) with
     $\XX^{\prime\prime}=\XX^{\prime}\setminus \{x_{t-k}\}$ to
     (\ref{eqn:two-first}) to conclude that for all $i$
  \begin{align*}
  \widetilde{H}_{i}(E_{k}\cap
  \langle{(F_{n-k})}_\XX^{c}\rangle)=\widetilde{H}_{i-1}(\langle{F_{t-k+1}},
   \dots,
   {F_{n-t-k-1}},
  {\{x_{n-t+1},\dots,x_{n-k-1}\}}\rangle^{c}_{\XX^{\prime\prime}}).
  \end{align*}
   Now we use Lemma~\ref{lem:lem3} with values $a=t-k+1$ and
    $s=n-2t-1=(p-2)(t+1)+d+1$ to conclude that
    \begin{align}
    \widetilde{H}_{i}(E_{k}\cap
   \langle{(F_{n-k})}_\XX^{c}\rangle)=\left\{
      \begin{array}{ll}
        K& d=0,\  i=2p-3\\
        K& d=t-k-1,\ i=2p-2\\
         0 & \mbox{otherwise}.
       \end{array}\right . \notag
    \end{align}
     \end{enumerate}
\end{proof}
\begin{theorem}\label{lem:newlem} Let $2\leq t \leq n$ and $\D=\D_t(C_n)$
be the path complex of a cycle $C_n$ with vertex set
$\XX=\{x_1,x_2,\dots, x_n\}$. Suppose $n=(t+1)p+d$ where $p\geq 0$,
$0\leq d \leq t$. Then for all $i$
$$\widetilde{H}_i(\D_\XX^c)=\left\{\begin{array}{ll}
K^t & d=0,\ i=2p-2, \ p>0\\
K & d \neq 0,\ i=2p-1\\
0 & \mbox{otherwise.}\\
\end{array}\right.$$
\end{theorem}
  \begin{proof} If $p=0$, then $n=t$ and our claim is obvious,
    so we assume that $p\geq 1$ and therefore $n\geq t+1$.

If we consider the simplicial complexes in (\ref{eqn:204}) then clearly 
 we have $\D_\XX^c=E_{t-1}$. We start with $E_0$ and apply the
   Mayer-Vietoris sequence repeatedly to calculate the homologies of
    the $E_k$.
    Since $E_0=E(n-t+1)$, we find
   $$n-t+1=(t+1)p+d-t+1= \left\{\begin{array}{ll}
      (t+1)p+1 & d=t\\
      (t+1)p & d=t-1\\
      (t+1)(p-1)+ d+2 & d<t-1
    \end{array}\right.$$
   which by Corollary~\ref{col:one-run} implies that
   \begin{eqnarray}\widetilde{H}_i(E_0)= \left\{
          \begin{array}{ll}
           K & d=0,\ i=2p-2\\
           K & d=t,\ i=2p-1\\
           0 & \mbox{otherwise.}
          \end{array}\right. \label{eqn:E0}
   \end{eqnarray}
 In order to find the homologies of $E_{t-1}$ we shall
   recursively apply the Mayer-Vietoris sequence as follows. For a
  fixed $1\leq k \leq t-1$ we have the following exact sequence
   \begin{eqnarray} \widetilde{H}_{i}(E_{k-1} \cap
   \langle
        {(F_{n-k+1})}_\XX^{c}\rangle)\rightarrow\widetilde{H}_i(E_{k-1})\rightarrow
        \widetilde{H}_i(E_{k})\rightarrow \widetilde{H}_{i-1}(E_{k-1}
        \cap \langle {(F_{n-k+1})}_\XX^{c}\rangle).\label{eqn:MV}
   \end{eqnarray}
\begin{enumerate}
 \item If $0 <d <t$ then by Lemma~\ref{lem:newlem1} we know that
   $\widetilde{H}_{i}(E_{k-1}\cap\langle{(F_{n-k+1})}_\XX^{c}\rangle)$
   is nonzero only when $i=2p-2$ and $k=t-d$. We apply this
   observation and~(\ref{eqn:E0}) to the exact sequence~(\ref{eqn:MV})
  to see that $$\widetilde{H}_i(E_{k})=\widetilde{H}_i(E_{k-1})=
   \widetilde{H}_i(E_{0})=0 \mbox{ for } 1 \leq k\leq t-d-1.$$ Once
  again we use~(\ref{eqn:MV}) to observe that
  \begin{align*}
  \widetilde{H}_i(E_{k})&=\left\{
     \begin{array}{lll}
       0 & 1 \leq k\leq t-d-1\\
      \widetilde{H}_{i-1}(E_{k-1}\cap \langle {(F_{n-k+1})}^{c}\rangle)& k=t-d\\
        \widetilde{H}_i(E_{t-d}) & t-d< k \leq t-1
      \end{array}\right. \notag \\
                      &=\left\{
      \begin{array}{lll}
        K& k\geq t-d, \ i=2p-1\\
        0& \mbox{otherwise}.
      \end{array}\right.
   \end{align*}
  We can conclude that in this case
   $$\widetilde{H}_i(E_{t-1})=\left\{
      \begin{array}{lll}
        K& i=2p-1\\
      0& \mbox{otherwise}.
     \end{array}\right. $$
  \item If $d=t$ then by Lemma~\ref{lem:newlem1} we know that
    $\widetilde{H}_{i}(E_{k-1}\cap\langle{(F_{n-k+1})}_\XX^{c}\rangle)$
    is always zero. We apply this fact along with~(\ref{eqn:E0}) to
    the sequence in (\ref{eqn:MV}) to observe
    that $$\widetilde{H}_i(E_{k})\cong\widetilde{H}_i(E_{0})=
       \left\{\begin{array}{ll}
         K &  i=2p-1\\
         0 & \mbox{otherwise}
              \end{array}\right.
     \ \ \mbox{ for } k\in\{1,2,\dots,t-1\}.  $$
\item If $d=0$ then by Lemma~\ref{lem:newlem1} we know that
  $\widetilde{H}_{i}(E_{k-1}\cap\langle{(F_{n-k+1})}_\XX^{c}\rangle)$
  is zero unless $i=2p-3$, and by (\ref{eqn:E0}) we know
  $\widetilde{H}_{i}(E_0)$ is zero unless $i=2p-2$. Applying these
  facts to~(\ref{eqn:MV}) we see
  that $$\widetilde{H}_{i}(E_k)=\widetilde{H}_{i}(E_0)=0\ \ \mbox{ for }i \neq 2p-2.$$
  When $i=2p-2$, the sequence~(\ref{eqn:MV}) produces an exact
  sequence $$0\longrightarrow
 \overbrace{\widetilde{H}_{2p-2}(E_{0})}^{K} \longrightarrow
  \widetilde{H}_{2p-2}(E_{1}) \longrightarrow
  \overbrace{\widetilde{H}_{2p-3}(E_{0}\cap \langle
    {(F_{n})}^{c}\rangle)}^{K}\longrightarrow 0.$$
Therefore $$\widetilde{H}_i(E_1)=\left\{
    \begin{array}{ll}
      K^2 & i=2p-2 \\
      0  & \mbox{otherwise.}
    \end{array} \right.$$

We repeat the above method, recursively , for values $k=2,3,\dots,t-1$
$$0 \longrightarrow \overbrace{\widetilde{H}_{2p-2}(E_{k-1})}^{K^k}
\longrightarrow \widetilde{H}_{2p-2}(E_{k})\longrightarrow
\overbrace{\widetilde{H}_{2p-3}(E_{k-1}\cap \langle
  {(F_{n-k+1})}^{c}\rangle)}^{K}\longrightarrow 0$$ and conclude that
for $1\leq k \leq t-1$
$$\widetilde{H}_i(E_k)=\left\{
   \begin{array}{ll}
     K^{k+1} & i=2p-2 \\
      0  & \mbox{otherwise.}
   \end{array} \right.$$
\end{enumerate}
We put this all together $$\widetilde{H}_i(E_{t-1})=\left\{
   \begin{array}{ll}
          K^t & d=0,\ i=2p-2,\ p>0 \\
          K   & d\neq 0,\ i=2p-1\\
          0  & \mbox{otherwise}
   \end{array}\right.$$ and this proves the statement of the theorem.
\end{proof}


\section{The Betti numbers}


We are now ready to apply the homological calculations from the
previous section to compute the top degree Betti numbers of path
ideals.  If $I$ is the degree $t$ path ideal of a cycle, then
\begin{align}
\beta_{i,j}(R/I)=0 \mbox{ for all }i\geq 1 \mbox{ and } j>ti,\label{e:gtti}
\end{align}
see for example~\cite{Jacques2004}~3.3.4.  By Theorem~\ref{col:col15},
to compute the Betti numbers of $I$ of degree $n$ we should consider $\D^{c}$.

\begin{theorem}[{\bf Betti numbers of degree $n$}]\label{theorem:theorem6.2}
Let $p$, $t$, $n$, $d$ be integers such that  $n=(t+1)p+d$, where $p\geq 0$, $0\leq d \leq t$, and $2\leq t\leq n$. If $C_n$ is a cycle over $n$ vertices, then
$$\beta_{i,n}(R/I_t(C_n))=\left\{\begin{array}{lll}
t && d= 0,    \ \displaystyle i=2\left (\frac{n}{t+1}\right )=2p\\
&&\\
1 &&d\neq 0, \ \displaystyle i=2\left (\frac{n-d}{t+1}\right )+1=2p+1 \\
&\hspace{.3in} &\\
0 && \mbox{otherwise}.
\end{array}
\right.$$
\end{theorem}
       \begin{proof}  Suppose $\D=\D_t(C_n)$. By Theorem~\ref{col:col15}
      $\beta_{i,n}(R/I_t(C_n))=\dim_{K}\widetilde{H}_{i-2}(\D_\XX^{c})$
         and the result now follows directly from
         Theorem~\ref{lem:newlem}.
       \end{proof}

From Hochster's formula we see that computing Betti numbers of degree
less than $n$ comes down to counting induced subcollections of certain
kinds. 

\begin{defn}\label{d:eligible} Let $i$ and $j$ be positive integers.
We call an induced subcollection $\Gamma$ of $\D_t(C_n)$ an {\bf
$(i,j)$-eligible subcollection} of $\D_t(C_n)$ if $\Gamma$ is
composed of disjoint runs of lengths
 \begin{eqnarray} (t+1)p_1+1,\dots, (t+1)p_{\alpha}+1, (t+1)q_1+2,
 \ldots, (t+1)q_{\beta}+2\label{eqn:length1}
 \end{eqnarray}
for nonnegative integers $\alpha, \beta,
p_1,p_2,\dots,p_{\alpha},q_1,q_2,\dots,q_{\beta}$, which satisfy the
following conditions
$$\begin{array}{lll} j&=&(t+1)(P+Q)+t(\alpha+\beta)+\beta
 \\ i&=&2(P+Q)+2\beta+\alpha,
  \end{array}$$
where $P=\sum_{i=1}^{\alpha} p_i$ and $Q=\sum_{i=1}^{\beta} q_i$.
\end{defn}

The next theorem is similar to a statement proved for the edge
ideal of a cycle in~\cite{Jacques2004}.

\begin{theorem}\label{lem:lem15} Let $I=I(\Lambda)$ be the
facet ideal of an induced subcollection $\Lambda$ of
$\D_t(C_n)$. Suppose $i$ and $j$ are integers with $i\leq
j<n$. Then the ${\mathbb N}$-graded Betti number $\beta_{i,j}(R/I)$
is the number of $(i,j)$-eligible  subcollections of $\Lambda$.
\end{theorem}
     \begin{proof} Since $\D(I)=\Lambda$ from Theorem~\ref{col:col15}
    we have $$\beta_{i,j}(R/I)={\displaystyle\sum_{\Gamma\subset
          \Lambda,|\ver (\Gamma)|=j}}\dim_K
      \widetilde{H}_{i-2}(\Gamma_{\ver (\Gamma)}^c)$$ where $\ver
      (\Gamma)$ is the vertex set of $\Gamma$ and the sum is taken
      over induced subcollections $\Gamma$ of $\Lambda$.

    Each induced subcollection of $\Lambda$ is clearly an induced
    subcollection of $\D_t(C_n)$, and can therefore be written as a
    disjoint union of runs. So from Proposition~\ref{prop:prop9}
   we can conclude the only $\Gamma$ whose complements have nonzero
  homology are those corresponding to run sequences of the form
   (\ref{eqn:length1}). Such subcollections have $j$ vertices where
    by Definition~\ref{defn:defn3.5}
   \begin{eqnarray} \nonumber j &=& ((t+1)p_1+t)+\dots+((t+1)p_{\alpha}+t)+
    ((t+1)q_{1}+t+1)\dots+((t+1)q_{\beta}+t+1)~\vspace{.1
        in}\\ &=&(t+1)(P+Q)+t(\alpha+\beta)+\beta. \label{eqn:IT}
\end{eqnarray}
  So $$\Gamma_{\ver
    (\Gamma)}^c=E((t+1)p_1+1,\dots,(t+1)p_{\alpha}+1,(t+1)q_1+2,\dots,
   (t+1)q_{\beta}+2)$$ and by Proposition~\ref{prop:prop13} we have
    \begin{eqnarray}\dim_{K}(\widetilde{H}_{i-2}
    (\Gamma_{\ver(\Gamma)}^c))= \left\{
      \begin{array}{ll}
     1 & i=2(P+Q)+2\beta+\alpha \\ 0 & \mbox{otherwise.}
    \end{array}\right. \label{eqn:ITT}
  \end{eqnarray}
   From (\ref{eqn:IT}) and (\ref{eqn:ITT}) we see that each induced
   subcollection $\Gamma$ corresponding to a run sequence as
  in~(\ref{eqn:length1}) contributes 1 unit to $\beta_{i,j}$ if
  and only if $$\begin{array}{lll}
     j&=&(t+1)(P+Q)+t(\alpha+\beta)+\beta
     \\ i&=&2(P+Q)+2\beta+\alpha.
\end{array}$$
\end{proof}

\begin{theorem}\label{nonzerobetinumbers} Let $i,j$ be integers and $i\leq j< n$ and suppose
 $n=(t+1)p+d$ and $d<t+1$. If $\beta_{i,j}(R/I_t(C_n))\neq 0$ we
 have 
\begin{enumerate}
\item $j\leq it$;
\item $j-i \leq (t-1)p$;
\item $i<2p$ \ \ if $d=0$; 
\item $i \leq 2p+1$ \ \ if $d \neq 0$. 
\end{enumerate}
\end{theorem}
\begin{proof} The fact that $j\leq it$ is just~(\ref{e:gtti}). 
By using Theorem~\ref{lem:lem15} we know $\beta_{i,j}(R/I_t(C_n)$ is
equal to the number of $(i,j)$-eligible subcollections of
$\Delta_t(C_n)$. So if we assume $\beta_{i,j}(R/I_t(C_n))\neq 0$ we
can conclude there exists a $(i,j)$-eligible subcollections $C$ of
$\Delta_t(C_n)$ which is composed of runs of lengths as described in
(\ref{eqn:length1}).  Therefore 
\begin{eqnarray}\label{repeatedequation}
j-i=(t-1)(P+Q+\alpha+\beta)& \mbox{ and }& ti-j=(t-1)(P+Q+\beta).
\end{eqnarray}
It follows that $j-i \geq ti-j$ so $$i(t+1) \leq 2j \Rightarrow i \leq 2 \left (\frac{j}{t+1} \right ) < 2 \left (\frac{p(t+1)+d}{t+1} \right )$$
so if $d=0$ it follows that $i<2p$ and if $d\neq 0$ it follows that $i\leq 2p+1$.

 On the other hand since $\Delta_t(C_n)$ has $n$ facets and since
 there must be at least $t$ facets between every two runs in $C$, we
 have
$$n\geq (t+1)P+(t+1)Q+\alpha+2\beta+t\alpha+t\beta \geq
(t+1)(P+Q+\alpha+\beta)=\left(\frac{t+1}{t-1}\right)(j-i)$$ which
implies that 
$$\frac{j-i}{t-1} \leq p+\frac{d}{t+1}$$ and since from
(\ref{repeatedequation}) we have $\frac{j-i}{t-1}$ is an integer, we must
have $\frac{j-i}{t-1} \leq p$. The inequality in (2) now follows.
\end{proof}

In~\cite{AF2012} we use combinatorial techniques to calculate the
remaining Betti numbers.

We end the paper with the computation of the projective dimension and
regularity of path ideals of cycles.  The case $t=2$ is the case of
graphs which appears in Jacques~\cite{Jacques2004}.

\begin{col}[{\bf Projective dimension and regularity of path ideals of cycles}]\label{c:pdr}
 Let $n$, $t$, $p$ and $d$ be integers such that $n\geq 2$, $2\leq t
\leq n$, $n=(t+1)p+d$, where $p\geq 0$, $0\leq d \leq t$. Then
 \renewcommand{\theenumi}{\roman{enumi}}
\begin{enumerate}
\item The projective dimension of the path ideal of a graph cycle
  $C_n$  is given by
$$pd(R/I_t(C_n))=
\left\{\begin{array}{ll}
2p+1& d \neq 0~\vspace{.1 in}\\
2p&d=0
\end{array}\right. \hspace{.5in}
$$
\item The regularity of the path ideal of the graph cycle
  $C_n$  is given by
$$reg(R/I_t(C_n))=\left\{\begin{array}{lll}
                       (t-1)p+d-1&\mbox{for}& d\neq 0 \\
                        (t-1)p&\mbox{for}& d=0 
                              \end{array}\right..$$
\end{enumerate}
\end{col}

  \begin{proof} \renewcommand{\theenumi}{\roman{enumi}}
    \begin{enumerate}
     \item This follows from Theorem~\ref{theorem:theorem6.2} and
       Theorem~\ref{nonzerobetinumbers}.
    \item By definition, the regularity of a module $M$ is $\max \{j-i
      \st \beta_{i,j}(M)\neq 0\}$. By Theorem~\ref{theorem:theorem6.2}
      and Theorem~\ref{nonzerobetinumbers} and the observation above
      if $d=0$ then $reg(R/I_t(C_n))$ is
    $$ \max \{n-2p,
    (t-1)p\}=\max\{(t+1)p-2p,(t-1)p\}=(t-1)p$$ and if $d\neq
    0$ then $reg(R/I_t(C_n))$ is 
     $$\max \{n-2p-1,
   (t-1)p\}=\max\{(t+1)p+d-2p-1,(t-1)p\}=(t-1)p+d-1.$$ The formula
    now follows.

    \end{enumerate}

  \end{proof}

\bibliographystyle{plain}
\bibliography{newpath}

\begin{thebibliography}{10}

\bibitem{AF2012}
A.~Alilooee and S.~Faridi.
\newblock Graded betti numbers of the path ideals of cycles and lines.
\newblock {\em arXiv:11106653[math.AC]}.

\bibitem{R.Bouchat2010}
R.~Bouchat, H.~T. Ha, and A.~O' Keefe.
\newblock Path ideals of rooted trees and their graded {B}etti numbers.
\newblock {\em J. Combin. Theory Ser. A}, 118(8):2411--2425, 2011.

\bibitem{Bruns1993}
W.~Bruns and J.~Herzog.
\newblock {\em Cohen-{M}acaulay rings}, volume~39 of {\em Cambridge Studies in
  Advanced Mathematics}.
\newblock Cambridge University Press, Cambridge, 1993.

\bibitem{Conca1999}
A.~Conca and E.~De Negri.
\newblock {$M$}-sequences, graph ideals, and ladder ideals of linear type.
\newblock {\em J. Algebra}, 211(2):599--624, 1999.

\bibitem{Macaulay2}
D.~R. Grayson and M.~E. Stillman.
\newblock Macaulay 2, a software system for research in algebraic geometry.
\newblock Available at http://www.math.uiuc.edu/Macaulay2/.

\bibitem{Hatcher2002}
A.~Hatcher.
\newblock {\em Algebraic topology}.
\newblock Cambridge University Press, Cambridge, 2002.

\bibitem{He2010}
J.~He and A.~Van Tuyl.
\newblock Algebraic properties of path ideal of a tree.
\newblock {\em Comm. Algebra}, 38:1725--1742, 2010.

\bibitem{Jacques2004}
S.~Jacques.
\newblock {\em Betti numbers of graph ideals}.
\newblock PhD thesis, University of Sheffield, June 2004.

\bibitem{Miller2005}
E.~Miller and B.~Sturmfels.
\newblock {\em Combinatorial commutative algebra}, volume 227 of {\em Graduate
  Texts in Mathematics}.
\newblock Springer-Verlag, New York, 2005.

\bibitem{CoCoA}
CoCoA Team.
\newblock Cocoa: a system for doing computations in commutative algebra.
\newblock Available at http://cocoa.dima.unige.it.

\bibitem{Villarreal2001}
R.~H. Villarreal.
\newblock {\em Monomial algebras}, volume 238 of {\em Monographs and Textbooks
  in Pure and Applied Mathematics}.
\newblock Marcel Dekker Inc., New York, 2001.

\end{thebibliography}

\end{document}